\documentclass[letterpaper, 10 pt, conference]{ieeeconf}  % Comment this line out
                                                          \usepackage{mathrsfs}
\usepackage{amsmath}
\usepackage{amssymb}
\newtheorem{theorem}{Theorem}
\newtheorem{lemma}{Lemma}

\newtheorem{assumption}{Assumption}

\newtheorem{corollary}{Corollary}
\newtheorem{problem}{Problem}
\usepackage{mathrsfs}
\usepackage{amsmath}
\usepackage{enumerate}
\usepackage{marvosym}

\IEEEoverridecommandlockouts                              % This command is only
   \usepackage{graphicx}                                                       % use the \thanks command
\title{\LARGE \bf
Synthesis of Interval Observers for Nonlinear Discrete-Time Systems
}

\author{Adam M Tahir and Beh\c{c}et A\c{c}\i kme\c{s}e 
\thanks{This work was supported by the United States Department of Defense SMART Scholarship. }
\thanks{A.M.T and B.A. are with 
the William E Boeing Department of Aeronautics and Astronautics, University of Washington, Seattle, WA,  98195, USA (e-mails: \{tahiram\},\{behcet\}@uw.edu)}}

\begin{document}

\maketitle
\thispagestyle{empty}
\pagestyle{empty}

%%%%%%%%%%%%%%%%%%%%%%%%%%%%%%%%%%%%%%%%%%%%%%%%%%%%%%%%%%%%%%%%%%%%%%%%%%%%%%%%
\begin{abstract}
A systematic procedure to synthesize interval observers for nonlinear discrete-time systems is proposed. The feedback gains and other matrices are found from the solutions to semidefinite feasibility programs. Two cases are considered: (1)  the interval observer is in the same coordinate frame as the given system, and (2) the interval observer uses a coordinate transformation. The conditions where coordinate transformations are necessary are detailed. Numerical examples are provided to showcase the effectiveness of the interval observers and demonstrate their application to sampled-data systems. 
\end{abstract}

\section{Introduction}
\label{sec:introduction}
%,  optimal control and path planning \cite{Discretization, TimeDisc}, and reachable set computation \cite{TimeDisc}
The state estimation of nonlinear discrete-time systems (NDTs) is widely studied problem. The motivation to study the state estimation of NDTs lies beyond the application to intrinsically discrete-time systems. Nonlinear continuous-time systems (NCTs)  are often discretized into NDTs for digital implementation \cite{astromwitenmark,disctaylor,approxobserver}.  Furthermore, models of NCTs are often given as NDTs from system identification \cite{continuousidentification}. There are numerous approaches to state estimation for NDTs, including:  extended Kalman filters \cite{EKFNonDisc}, unscented Kalman filters \cite{UnscentedDT},  high-gain observers \cite{high_gainobsdisc}, linearization methods \cite{linearization,linearizableerror,ObsDT}, moving horizon estimators \cite{movinghorizon}, and Luenberger observers \cite{ObserverLipDisc,discretetimeluenberger,UnknownInputObs,pigeons, luenbergerdt}. 
In contrast to the aforementioned state estimators, which provide point estimates of the state, an interval observer (IO) provides a compact set to which the state of a system belongs at each instance of time. The state is enveloped in an interval by ensuring positivity of the errors between the state and its upper and lower bounds. In addition, input-to-state (ISS) stability with respect to the disturbance is ensured so that the interval is ultimately bounded in proportion to the size of the disturbance. Finding IOs that achieve both positivity and ISS stability in the same coordinate frame as the system being estimated is not always possible. In these scenarios, a change of coordinates can often be found where positivity and ISS stability follow more readily.  

IOs for linear discrete-time systems \cite{DiscreteIO,  IODT, UncertainDiscreteTime, Outbreakobserver} and NCTs \cite{Polytopic,IntervalEstimationNonlin, IntervalLip, TahirThesis} have been widely studied in the literature. Less attention has been given to NDTs.   \cite[\S III]{IODT} provides generic conditions under which IOs can be constructed, but no systematic procedure to construct them. Besides, the nonlinear example in that paper is a stable one-dimensional system with no output. In \cite[Ch. 3]{TahirThesis}, the approach in \cite{UncertainDiscreteTime} is extended to synthesize IOs for polytopic NDTs using the linear parameter-varying framework.

In this paper, IOs for NDTs with nonlinearities that have bounded Jacobians are synthesized from the solutions to semidefinite feasibility programs.  The proposed IOs use a linear output feedback term and a mixed-monotone decomposition\footnote{A decomposition of a function into its increasing and decreasing elements. See\cite{CooganFinitieAbstraction,SuffMixMonotone} for a formal definition.}  of the nonlinearity with an output feedback term injected into it. The feasibility programs consist of a set of linear inequalities and a linear matrix inequality (LMI). The linear inequalities ensure positivity of the linear part of the error dynamics and construct the mixed-monotone decomposition, which ensures positivity of the nonlinear part of the error dynamics. The LMI ensures ISS stability. 

The use of injection feedback terms and their synthesis from LMI problems has become commonplace in Luenberger observers for NCTs and NDTs \cite{circlecriterion,DQCObserver, UnknownInputObs, PNLV, Outputfeedback,pigeons}. Their use and synthesis in IOs has not been studied before, to the best of the authors knowledge.  Standard approaches from the Luenberger observer literature cannot directly be applied to the IO synthesis problem because the synthesis of Luenberger observers is exclusively dedicated to the stability of the error terms. In IOs, positivity and stability need to be achieved simultaneously.  Based on the linear inequalities that the mixed-monotone decomposition is constructed from,  it is shown that the mixed-monotone decomposition satisfies an incremental quadratic constraint ($\delta$QC) which can be easily incorporated into the LMI while being able to solve for the injection feedback terms.

 %sBy using a mixed-monotone decomposition of the nonlinearity instead of a simple copy of the nonlinearity,  positivity of the nonlinear part of the error dynamics can be ensured. Moreover, that easily be incorporated into the LMI for ISS stability while 

%In this paper, an i is constructed for the mixed-monotone decomposition that can be used to directly solve for the injection feedback terms in the LMI, and linear inequalities are constructed to ensure

The contribution of this paper is the proposition of a systematic procedure to synthesize IOs for NDTs. Specifically:  
\begin{enumerate}[i.]
\item A semidefinite feasibility program to solve for the observer gains and coupling matrices is proposed.
\item To the best of the authors knowledge, this is the first result where injection feedback terms for IOs are synthesized. Numerical examples are provided to showcase their advantages.
\item Conditions where coordinate transformations are necessary are detailed. 
\item A procedure to synthesize IOs with coordinate transformation is provided. 
\end{enumerate}

This paper is organized as follows: \S \ref{s:problemsetup} details the structure of the NDTs under consideration, formally defines the objectives of the IOs, and provides some useful lemmas pertaining to interval analysis. \S\ref{s:IO} provides the structures of the IOs. \S\ref{s:synth} and \S\ref{s:synthcoord} detail the synthesis of the IOs without and with coordinate transformation, respectively. \S\ref{s:numex} provides an example showcasing the advantages of using the injection feedback terms, and a sampled-data system example.  
\subsection{Notation}
$\mathbb{R}$ denotes the set of real numbers, $\mathbb{R}_{\ge 0}$ denotes the set of nonnegative real numbers, $\mathbb{Z}_{\ge 0}$ denotes the set of nonnegative integers. $\|\cdot\|$ denotes the Euclidean norm. For a function $\phi:\mathbb{Z}_{\ge 0}\to \mathbb{R}^{n}$, the notation $\|\phi\|_{\ell_\infty}=\sup\{\|\phi[k]\|: k\in\mathbb{Z}_{\ge 0}\}$, which is the standard $\ell_\infty$-norm when $\phi$ is bounded.

For a matrix $A\in\mathbb{R}^{n\times m}$,  $A\ge 0$ means that all the entries of $A$ are nonnegative. $A^\oplus$ is defined as $A_{ij}^\oplus = \max\{0, A_{ij}\}$ and $A^\ominus=A^\oplus-A$.  The transpose of $A$ is denoted by $A^\top$. For two matrices $A,B\in\mathbb{R}^{n\times m}$, $A\ge B\Leftrightarrow A-B\ge 0$.   The $n\times n$ identity matrix is denoted by $I_n$. %The vector of all ones of dimension $n$ is denoted by ${\bf 1}_n$, and an $n\times m$ matrix of all ones is denoted by ${\bf 1}_{n\times m}$.  The Kronecker product of a matrix $A$ with a matrix $B$ is denoted by $A\otimes B$. \

For a matrix $P=P^\top\in\mathbb{R}^{n\times n}$, the notation $P\succ 0$ (resp. $\preceq 0$) denotes $P$ is positive definite  (resp.  negative semidefinite). The $\star$ notation denotes matrix entries that follow from symmetry. %For a matrix $A\in\mathbb{R}^{n\times m}$, %Note that $A^\oplus$ and $A^\ominus$ are nonnegative matrices which consist of the positive entries of $A$ with the rest zeros and the negative entries of $A$ with the rest zeros, respectively.

A matrix $A\in\mathbb{R}^{n\times n}$ is Schur if all of its eigenvalues belong to the unit disk. It is an M-matrix if its diagonal entries are strictly positive and all of its off-diagonal entries are nonpositive. The spectrum of $A$ is denoted by $\sigma(A)$

\section{Preliminaries}\label{s:problemsetup}
\subsection{System Description}
This paper considers NDTs with state $x\in\mathbb{R}^{n}$ and output $y\in\mathbb{R}^m$ that have the following structure:
\begin{subequations}\label{e:sysNDT}
\begin{align}
&x[k+1] = Ax[k]+p(x[k])+w[k],\\
&y[k] = Cx[k]. 
\end{align}
\end{subequations}
The system matrices are $A\in\mathbb{R}^{n\times n}$ and $C\in\mathbb{R}^{m\times n}$. The nonlinearity $p$ is assumed to have a globally bounded Jacobian that is characterized in the following:
\begin{assumption}\label{assp}
The nonlinearity $p:\mathbb{R}^{n}\to\mathbb{R}^{n}$ is differentiable and there exist known matrices $\underline{D}, \overline{D}$ where $ \overline{D}\ge 0$ and $\underline{D}\le 0$ such that $\underline{D}\le \frac{\partial p}{\partial x}(x)\le \overline{D}$ for all $x\in\mathbb{R}^{n}.$
\end{assumption}
 The characterization of $p$ by linear inequalities lends itself naturally to the proposition of linear inequalities which ensure positivity of the error dynamics. Moreover, the characterization includes the class of differentiable globally Lipschitz nonlinearities and polytopic nonlinearities. 

The disturbance $w$ is unknown and bounded. The following assumption on the disturbance  is standard in the IO literature:
\begin{assumption}\label{assvw}
The vectors $\underline{w},\overline{w}\in\mathbb{R}^n$ are known, and are such that $\underline{w}\le w[k]\le \overline{w}$ for all $ k\in\mathbb{Z}_{\ge 0}$.
\end{assumption}

\subsection{Interval Observer Objectives}\label{s:objectifs}
An IO consists of two state estimates $\underline{x}[k], \overline{x}[k]\in\mathbb{R}^n$. Define the following error term:
\begin{align}
\varepsilon[k] =\left[\begin{array}{cc}\overline{x}^\top[k]-x^\top[k] & x^\top[k]-\underline{x}^\top[k]\end{array}\right]^\top\label{edef}.
\end{align} 
The dynamics of $\underline{x}$ and $\overline{x}$ are to be designed using output feedback so that the dynamics of $\varepsilon$ satisfy the following two objectives:
\begin{enumerate}
\item {\bf Positivity.} If $\varepsilon[0]\ge 0$, then $\varepsilon[k]\ge 0$ for all $k\in\mathbb{Z}_{\ge 0}$.
\item {\bf Input-to-State Stability.} The error is bounded as follows:
\begin{align}
\left\|\varepsilon[k]\right\|\le \beta\left(\|\varepsilon[0]\|, k\right)+\rho\left(\|\Delta w\|_{\ell_\infty}\right),\label{e:ISSgoal}
\end{align}
 where  $\beta\in\mathcal{KL}$, $\rho\in\mathcal{K}_\infty$,\footnote{The definitions of class $ \mathcal{K}_\infty$ and $\mathcal{KL}$ functions can be found in  
\cite[\S 2]{ISSdiscrete}.} and
\begin{align}
\Delta w[k]=\left[\begin{array}{cccc} \overline{w}^\top-w^\top[k]& w^\top[k]-\underline{w}^\top \end{array} \right]^\top,\label{e:deltawdef}
\end{align}
for all $k\in\mathbb{Z}_{\ge 0}$.
\end{enumerate}

The error $\varepsilon[k]$ is defined in \eqref{edef} in such a way that the positivity property implies that if $\underline{x}[0]\le x[0]\le \overline{x}[0]$, then $\underline{x}[k]\le x[k]\le \overline{x}[k]$ for all $k\in\mathbb{Z}_{\ge 0}$. This means that if the initial condition is known to belong to some compact interval, then for all $k\in\mathbb{Z}_{\ge 0}$, an interval to which $x[k]$ belongs is known. The ISS property ensures that this interval remains bounded, and the ultimate bound is proportional to $\|\Delta w\|_{\ell_\infty}$. 
\subsection{Useful Interval Analysis Lemmas}
Before proceeding to the main results, two useful lemmas pertaining to interval analysis are presented. The first lemma is a standard result, so its proof is omitted.

\begin{lemma}\label{l:mmp}
Consider matrices $\overline{B},\underline{B}\in\mathbb{R}^{n\times m}$ where $\underline{B}\le \overline{B}$, and a matrix $A\in\mathbb{R}^{l\times n}$.  The following holds: $A^\oplus\underline{B}-A^\ominus\overline{B}\le AB\le A^\oplus\overline{B}-A^\ominus\underline{B}$ for all $B\in\mathbb{R}^{n\times m}$ such that $\underline{B}\le B\le \overline{B}$.
\end{lemma} 

The following is a new result which will be used to find linear inequalities that construct the mixed-monotone decomposition of $p$ and the $\delta$QC that it satisfies.

%ensure positivity of the error dynamics and to construct a quadratic characterization of the nonlinearities with the nonlinear injection terms.
\begin{lemma}\label{l:fustercluck}
Consider matrices $\overline{A},\underline{A},\overline{B},\underline{B}\in\mathbb{R}^{n\times n}_{\ge 0}$.  The following holds:
\begin{align}
-\underline{A}\overline{B}-\overline{A}\underline{B}\le AB\le \overline{A}\hspace{1pt}\overline{B}+\underline{A}\hspace{1pt}\underline{B},\label{e:fuster}
\end{align}
for all $-\underline{A}\le A\le \overline{A}$ and $-\underline{B}\le B\le \overline{B}$. 
\end{lemma}
\begin{proof}
Decompose $AB=(A^\oplus-A^\ominus)(B^\oplus-B^\ominus)=(A^\oplus B^\oplus+A^\ominus B^\ominus)-(A^\ominus B^\oplus + A^\oplus B^\ominus)$. Recall that $A^\oplus, A^\ominus, B^\oplus, B^\ominus\ge 0$. Therefore, $-(A^\ominus B^\oplus + A^\oplus B^\ominus)\le AB\le (A^\oplus B^\oplus+A^\ominus B^\ominus).$ Since $-\underline{A}\le A^\oplus-A^\ominus\le \overline{A}$, it follows that $A^\oplus\le \overline{A}$ and $A^\ominus\le \underline{A}$. Similarly, $B^\oplus\le \overline{B}$ and $B^\ominus\le \underline{B}$. Therefore, $A^\oplus B^\oplus\le \overline{A}\hspace{1pt}\overline{B}$, and $A^\ominus B^\ominus\le \underline{A}\hspace{1pt}\underline{B}$, so the right-hand inequality in \eqref{e:fuster} follows. The left-hand inequality follows with the same logic. 
\end{proof}

\section{Proposed Interval Observers}\label{s:IO}
Consider the following IO for the NDT \eqref{e:sysNDT}:
\begin{subequations}\label{e:intobs}
\begin{align}
\overline{x}[k+1] = &(A-LC)\overline{x}[k]+\pi(\overline{x}[k], \underline{x}[k],y[k] )+Ly[k]\nonumber\\
&+F(\overline{x}[k]-\underline{x}[k])+\overline{w},\\
\underline{x}[k+1] = &(A-LC)\underline{x}[k]+\pi(\underline{x}[k], \overline{x}[k], y[k])+Ly[k]\nonumber\\
&+F(\underline{x}[k]-\overline{x}[k])+\underline{w}.
\end{align}
\end{subequations}
where $L\in\mathbb{R}^{n\times m}$, $F\in\mathbb{R}^{n\times n}$,
\begin{align*}
&\pi\left(x_1,x_2, y\right)=p((I_n-KC)x_1+Ky)+G\left(x_1-x_2\right),
%&\underline{p}\left(\overline{x},\underline{x}, y\right)=p\left((I_n-KC)\underline{x} +Ky)\right)+G\left(\underline{x}-\overline{x}\right),
\end{align*}
$K\in\mathbb{R}^{n\times m}$, and $G\in\mathbb{R}^{n\times m}$. The matrix $ L$ is the linear feedback gain, and the matrix $K$ is the nonlinear injection feedback gain. The matrices $F$ and $G$ are referred to as the coupling matrices, which are used for ensuring positivity.  The matrices $K$ and $G$ are to be designed such that $\pi$ is a mixed-monotone decomposition function for the nonlinearity $p$.

The problem to be solved of synthesizing the IO \eqref{e:intobs} is described as follows:

\begin{problem}\label{p:leafsforever}
Compute the gains $L,K\in\mathbb{R}^{n\times m}$ and coupling matrices $F, G\in\mathbb{R}^{n\times n}$ for \eqref{e:intobs} such that the dynamics of $\varepsilon$, defined in  \eqref{edef}, is a positive system and \eqref{e:ISSgoal} holds.
\end{problem}
\subsection{Interval Observer with Coordinate Transformation}
In some cases, the IO \eqref{e:intobs} cannot be synthesized because the necessary conditions for positivity and ISS stability of the error dynamics cannot be achieved in the given coordinate system $x$.  In these cases, the NDT \eqref{e:sysNDT} can be rewritten in a new coordinate system $z=Sx$, where $S\in\mathbb{R}^{n\times n}$ is an invertible coordinate transformation matrix. The proposed IO with the coordinate transformation is the following:
\begin{subequations}\label{e:intobscoord}
\begin{align}
\overline{z}[k+1] = &S(A-\Lambda C)U\overline{z}[k]+\tilde{\pi}(\overline{z}[k],\underline{z}[k],y[k])+S\Lambda y[k] \nonumber\\
&+\Phi(\overline{z}[k]-\underline{z}[k])+S^\oplus \overline{w}-S^\ominus\underline{w}\\
\underline{z}[k+1] = &S(A-\Lambda C)U\underline{z}[k]+\tilde{\pi}(\underline{z}[k],\overline{z}[k],y[k])+S\Lambda y[k] \nonumber\\
&+\Phi(\underline{z}[k]-\overline{z}[k])+S^\oplus \underline{w}-S^\ominus\overline{w}.
\end{align}
where 
\begin{align*}
&\tilde{\pi}\left(z_1,z_2, y\right)=Sp((U-HCU)z_1 +Hy)+\Gamma\left(z_1-z_2\right),
%&\underline{\pi}\left(\overline{z},\underline{z}, y\right)=Sp\left((U-HCU)\underline{z} +Hy)\right)+\Gamma\left(\underline{z}-\overline{z}\right),
\end{align*}
\end{subequations}
$\Lambda,H\in\mathbb{R}^{n\times m}$ are the linear and nonlinear injection gains, $\Phi,\Gamma\in\mathbb{R}^{n\times n}$ are the linear and nonlinear coupling matrices, and $S=U^{-1}$.
The new coordinate system, observer gains, and coupling matrices are to be chosen such that the  error 
\begin{align}
\xi[k] =\left[\begin{array}{cc}\overline{z}^\top[k]-z^\top[k] & z^\top[k]-\underline{z}^\top[k]\end{array}\right]^\top\label{xidef}
\end{align} 
is positive and ISS.  
\begin{problem}\label{p:coord}
Find $S, \Phi, \Gamma \in\mathbb{R}^{n\times n}$ and $\Lambda, H\in\mathbb{R}^{n\times m}$ for \eqref{e:intobscoord} such that the dynamics of $\xi$, defined in  \eqref{xidef}, is a positive system and ISS with respect to $\Delta w$, defined in \eqref{e:deltawdef}.
\end{problem}

\section{Interval Observer Synthesis}\label{s:synth} 

In this section a semidefinite program to solve Problem \ref{p:leafsforever} is presented. The dynamics of the error $\varepsilon$, defined in \eqref{edef}, for \eqref{e:intobs} are as follows:
\begin{align}
\varepsilon[k+1] = \mathcal{A}\varepsilon[k]+\Delta p[k]+\Delta w[k],\label{e:errordyn}
\end{align}
where
\begin{align}
&\mathcal{A}=\small \left[\begin{array}{cc}
A-LC+F & F\\
F & A-LC+F
\end{array}\right]\normalsize,\label{e:A}\\
&\Delta p[k]=\small \left[\begin{array}{cc}
\pi\left(\overline{x}[k],\underline{x}[k], y[k]\right)-p(x[k])\\
p(x[k])-\pi\left(\underline{x}[k],\overline{x}[k], y[k]\right)
\end{array}\right],\normalsize\nonumber
\end{align}
and $\Delta w[k]$ is defined in \eqref{e:deltawdef}. 

Positivity of \eqref{e:errordyn} follows if $\mathcal{A}$ is nonnegative, $\Delta p[k]\ge 0$ and $\Delta w[k]\ge 0$ for all $k\in\mathbb{Z}_{\ge 0}$. The linear gain $L$ and coupling matrix $F$ are to be chosen such that $\mathcal{A}$ is nonnegative. The term $\Delta p[k]\ge 0$ for all $k\in\mathbb{Z}_{\ge 0}$ if the injection gain $K$ and coupling matrix $G$ are such that 
\begin{align}
\pi( \underline{x},\overline{x}, Cx)\le p(x)= \pi( x,x, Cx)\le \pi(\overline{x},\underline{x},  Cx), \label{e:mixedmonotnocity}
\end{align}
for all $\underline{x}\le  x\le \overline{x}$, meaning that $\pi$ is a mixed-monotone decomposition function for $p$  \cite{SuffMixMonotone, CooganFinitieAbstraction}. Since $p$ has a globally bounded Jacobian it is mixed-monotone globally in $\mathbb{R}^n$, and $K, G$ such that  \eqref{e:mixedmonotnocity} holds can always be found by the satisfaction of certain linear inequalities involving the bounds on the Jacobian $\underline{D}$ and $\overline{D}$ from Assumption \ref{assp} (cf. \cite[Theorem 2]{SuffMixMonotone}). The disturbance term $\Delta w[k]\ge 0$ for all $k\in\mathbb{Z}_{\ge 0}$ as a consequence of Assumption \ref{assvw}.

ISS stability of \eqref{e:errordyn} follows by finding an ISS-Lyapunov function, which can be performed by feasibility of an LMI. A necessary condition for ISS stability is that $\mathcal{A}$ is Schur. %It will be shown that, when $\pi$ is a mixed-monotone decomposition function, $\Delta p[k]$ satisfies a $\delta$QC of the form:
%
%  for all $k\in\mathbb{Z}_{\ge 0}$ where $M_{22}=M_{22}^\top\in\mathbb{R}^{2n\times 2n}$ and $M_{12}\in\mathbb{R}^{2n\times 2n}$. The matrices $M_{12}$ and $M_{22}$ will be linearly dependent on $K$ and $G$. %The $\delta$QC  will be well-suited to be able to directly solve for $K$ and $G$ via LMIs in tandem with the linear inequalties. % (cf. \cite[\S 4.2]{DQCObserver},  \cite[\S VI]{UnknownInputObs}, \cite[\S III.B]{Outputfeedback}, and \cite[\S II.C]{PNLV}). 

\begin{theorem}\label{t:syntheorem}
Consider the system \eqref{e:sysNDT} where the nonlinearity $p$ satisfies Assumption \ref{assp} and Assumption \ref{assvw} holds. Suppose there exist\footnote{The bold variables denote variables that are solved for.} an M-matrix ${\bf J}\in\mathbb{R}^{n\times n}$, matrices ${\bf Y,K}\in\mathbb{R}^{n\times m}$, ${\bf W,\underline{\Upsilon},\overline{\Upsilon},G}\in\mathbb{R}^{n\times n}_{\ge 0}$, positive definite matrix ${\bf P}={\bf P}^\top\in\mathbb{R}^{2n\times 2n}$, scalars $\gamma, \tau >0$ and $\lambda\in[0,1)$ such that the following is feasible:
\begin{subequations}\label{synth}
\begin{align}
\mathcal{Q}\ge 0,\label{e:Qpos}\\
-{\bf \underline{\Upsilon}}\le I_n-{\bf K}C\le {\bf \overline{\Upsilon}},\label{e:boundsups}\\
%-{\bf \underline{\Upsilon}}\le I_n-{\bf K}C\le {\bf \overline{\Upsilon}},\label{e:boundsdowns}\\
\underline{D}{\bf \overline{\Upsilon}}-\overline{D}\hspace{1pt}{\bf \underline{\Upsilon}}+{\bf G}\ge 0,\label{e:posups}\\
\small 
\left[\begin{array}{cccc}
-\lambda {\bf P} & \mathcal{Q}^\top & \frac{\tau}{2}\Psi^\top & 0\\
\star & {\bf P}-\mathcal{J}-\mathcal{J}^\top & \mathcal{J} & \mathcal{J}\\
\star & \star & -\tau I_{2n} & 0\\
\star & \star & \star & -\gamma I_{2n}
\end{array}\right]\normalsize\preceq 0,\label{e:LMI1}
\end{align}
\end{subequations}
where $\mathcal{J}=\text{blkdiag}\left({\bf {J},{J}}\right)$,
\begin{align}
\mathcal{Q} = \small \left[\begin{array}{cc}
{\bf J}A-{\bf Y}C+{\bf W} & {\bf W}\\ 
{\bf W} & {\bf J}A-{\bf Y}C +{\bf W}
\end{array}\right],\normalsize\nonumber\\
\Psi = \small \left[\begin{array}{cc}
\overline{D}\hspace{1pt}{\bf \overline{\Upsilon}}-\underline{D}\hspace{1pt}{\bf \underline{\Upsilon}}+{\bf G }& {\bf G}\\
{\bf G} &\overline{D}\hspace{1pt}{\bf \overline{\Upsilon}}-\underline{D}\hspace{1pt}{\bf \underline{\Upsilon}}+{\bf G}
\end{array}\right].\normalsize\label{e:psimm}
\end{align}
The IO \eqref{e:intobs} where $L={\bf J}^{-1}{\bf Y}, \;F={\bf J}^{-1}{\bf W}, \;K={\bf K},  \;G={\bf G}, $
 and the initial conditions are such that $\underline{x}[0]\le x[0]\le \overline{x}[0]$, satisfies the following:
\begin{enumerate}[i.]
\item $\underline{x}[k]\le x[k]\le \overline{x}[k]$ for all $k\in\mathbb{Z}_{\ge 0}$. 
\item There exist $\beta\in\mathcal{KL}$ and $\rho\in\mathcal{K}_\infty$ such that \eqref{e:ISSgoal} holds.
\end{enumerate}
\end{theorem}
\begin{proof}
There are three parts to this proof. This first part shows that the satisfaction of  linear inequalities \eqref{e:Qpos}-\eqref{e:posups} and Assumption \ref{assvw} are sufficient for \eqref{e:errordyn} to be a positive system, so the first assertion holds. Then, it is shown that the since $\varepsilon[k]\ge 0$ for all $k\in\mathbb{Z}_{\ge 0}$, $\Delta p[k]$ satisfies a $\delta$QC in the form of \eqref{e:quadconstraint}, which is useful for showing that the matrix inequality \eqref{e:LMI1} implies the positive-definite function 
\begin{align}
V(\varepsilon)=\varepsilon^\top {\bf P} \varepsilon\label{e:Lyap}
\end{align}
is an ISS-Lyapunov function, so the second assertion holds. 
\paragraph{Positivity} The matrix $\mathcal{Q}=J\mathcal{A}$, where $\mathcal{A}$ is defined in \eqref{e:A}, with the following variable substitutions: ${\bf Y}={\bf J}\hspace{2pt}L$ and ${\bf W}={\bf J}\hspace{2pt}F$. Since ${\bf J}$ is an M-matrix, $\mathcal{J}$ is also an M-matrix.  The feasibility of \eqref{e:LMI1} implies that $\mathcal{J}+\mathcal{J}^\top\succeq P\succ 0$. Therefore, $\mathcal{J}$ is invertible. The inverse of an M-matrix is a nonnegative matrix \cite[pp. 161-163]{Greenbaum}. Hence $\mathcal{J}^{-1}\ge 0$, which, along with \eqref{e:Qpos}, implies that 
\begin{align}
\mathcal{A}=\mathcal{J}^{-1}\mathcal{Q}\ge 0.\label{e:posA}
\end{align}

By Assumption \ref{assvw}, 
\begin{align}
\Delta w[k]\ge 0, \label{e:deltaw}
\end{align}
for all $k\in\mathbb{Z}_{\ge 0}$. 

By the differential mean-value theorem \cite[\S 4.3]{LMIbook}, for every $k\in\mathbb{Z}_{\ge 0}$, there exist matrices $D_1[k], D_2[k]$ such that  
\begin{align}
\Delta p[k]=\Omega[k]\varepsilon[k],\label{e:deltapomega}
\end{align}
where 
\begin{align}
&\Omega[k]= \small \left[\begin{array}{cc}
D_1[k](I_n-KC)+G & G\\
G & D_2[k](I_n-KC)+G 
\end{array}\right],\normalsize\nonumber
\end{align}
and
\begin{align}
&\underline{D}\le D_1[k], D_2[k]\le \overline{D}.\label{e:Dinterval}
\end{align}

Using Lemma \ref{l:fustercluck}, \eqref{e:boundsups}, and \eqref{e:Dinterval}, it follows that 
\begin{align}
\underline{D}{\bf \overline{\Upsilon}}-\overline{D}\hspace{1pt}{\bf \underline{\Upsilon}}\le D_1[k](I_n-{\bf K}C)\le \overline{D}\hspace{1pt}{\bf \overline{\Upsilon}}-\underline{D}\hspace{1pt}{\bf \underline{\Upsilon}}\label{e:sandwichups}
\end{align}
for all $k\in\mathbb{Z}_{\ge 0}$. Therefore,  \eqref{e:posups} implies that $D_1[k](I_n-{\bf K}C)+{\bf G}\ge 0$, meaning that $D_1[k](I_n-KC)+G\ge 0$ for all $k\in\mathbb{Z}_{\ge 0}$. Similarly,  $D_2[k](I_n-KC)+G \ge 0$ for all $k\in\mathbb{Z}_{\ge 0}$. In addition, $G={\bf G}\ge 0$, so
\begin{align}
\Omega[k]\ge 0\label{e:Omega}
\end{align}
for all $k\in\mathbb{Z}_{\ge 0}$. Plugging \eqref{e:deltapomega} into \eqref{e:errordyn} yields:
\begin{align}
\varepsilon[k+1]=(\mathcal{A}+\Omega[k])\varepsilon[k]+\Delta w[k].\label{e:LTV}
\end{align}
 Since \eqref{e:posA}, \eqref{e:deltaw}, and \eqref{e:Omega} hold for all $k\in\mathbb{Z}_{\ge 0}$, it can be clearly seen from \eqref{e:LTV} that if $\varepsilon[k]\ge 0$ for some $k\in\mathbb{Z}_{\ge 0}$, then $\varepsilon[k+1]\ge 0$. The initial conditions satisfy $\underline{x}[0]\le x[0]\le \overline{x}[0]$, so $\varepsilon[0]\ge 0$. Therefore, by induction, $\varepsilon[k]\ge 0$ for all $k\in\mathbb{Z}_{\ge 0}$.  This proves the first assertion. 
 
 \paragraph{Quadratic Characterization of $\Delta p$}
Since $\varepsilon[k]\ge 0$ and \eqref{e:Omega} holds for all $k\in\mathbb{Z}_{\ge 0}$, it follows from \eqref{e:deltapomega} that 
\begin{align}
\Delta p[k]\ge 0\label{e:quad1}
\end{align}
for all $k\in\mathbb{Z}_{\ge 0}$. Using \eqref{e:sandwichups}, it follows that $\Omega[k]\le \Psi$ for all $k\in\mathbb{Z}_{\ge 0}$, where $\Psi$ is defined in \eqref{e:psimm}. Since, in addition, $\varepsilon[k]\ge 0$, 
\begin{align}
\Delta p[k]\le \Psi\varepsilon[k]\label{e:quad2}
\end{align}
for all $k\in\mathbb{Z}_{\ge 0}$. The combination of \eqref{e:quad1} and \eqref{e:quad2} imply that $(\Psi\varepsilon[k]-\Delta p[k])^\top\Delta p[k]\ge 0,$
 which can be written as the $\delta$QC
 
  \begin{align}
\left[\begin{array}{cc}\varepsilon[k]\\ \Delta p[k]\end{array}\right]^\top\left[\begin{array}{cc}
 0 & \frac{1}{2}\Psi^\top \\ \star & -I_{2n}
 \end{array}\right]\left[\begin{array}{cc}\varepsilon[k]\\ \Delta p[k]\end{array}\right]\ge 0\label{e:quadconstraint}
 \end{align}
 for all $k\in\mathbb{Z}_{\ge 0}$. 
%This will be used to construct the LMI for ISS stability. 
  
  \paragraph{ISS Stability}
  Consider the quadratic function \eqref{e:Lyap}, which evolves as follows:
  \begin{align*}
 &V(\varepsilon[k+1])-  V(\varepsilon[k])= \varepsilon^\top[k]\left(\mathcal{A}^\top {\bf P}\mathcal{A}-{\bf P}\right)\varepsilon[k]\\
 &+2\varepsilon^\top[k]\mathcal{A}^\top{\bf P} \Delta p[k]+2\varepsilon^\top[k]\mathcal{A}^\top{\bf P}\Delta w[k]+\Delta p^\top[k]{\bf P}\Delta p[k]\\
 &+2\Delta p^\top[k]{\bf P}\Delta w[k]+\Delta w^\top[k]{\bf P}\Delta w[k].
  \end{align*}
Since the $\delta$QC \eqref{e:quadconstraint} is satisfied, by the $\mathcal{S}$-procedure \cite[\S 2.6.3]{LMIbook}, if the following holds:
\begin{align}
\small \begin{bmatrix}
\mathcal{A}^\top {\bf P} \mathcal{A}-\lambda {\bf P} & \mathcal{A}^\top {\bf P}+\frac{\tau}{2}\Psi^\top & \mathcal{A}^\top {\bf P}\\
\star & {\bf P}-\tau I_{2n} & {\bf P}\\
\star & \star & {\bf P}-\gamma I_{2n}
\end{bmatrix}\preceq 0\normalsize,\label{e:pigeonLMI}
\end{align}
then 
\begin{align}
 V&(\varepsilon[k+1])\le \lambda  V(\varepsilon[k])+\gamma \|\Delta w[k]\|\label{e:ISSblah}
\end{align}
for all $k\in\mathbb{Z}_{\ge 0}$.  By the Projection Lemma \cite{LMIcond}, \eqref{e:pigeonLMI} is equivalent to \eqref{e:LMI1} with the variable substitution $\mathcal{Q}=\mathcal{J}\mathcal{A}$. Since $\lambda\in[0,1)$, \eqref{e:ISSblah} implies that  $V$ is an ISS-Lyapunov function, and the error dynamics \eqref{e:errordyn} are ISS with respect to $\Delta w$ \cite[Theorem 1]{ISSdiscrete}. Hence, the satisfaction of the \eqref{e:LMI1} implies that  the error dynamics \eqref{e:errordyn} are ISS with respect to $\Delta w$. Consequently, the second assertion holds. 
\end{proof}

The inequalities \eqref{e:Qpos}-\eqref{e:posups} are linear in the solution variables. Moreover, the constraints that  ${\bf W,\underline{\Upsilon},\overline{\Upsilon},G}$ are nonnegative and ${\bf J}$ is an M-Matrix are linear inequalities. The matrix inequality \eqref{e:LMI1} becomes an LMI when $\tau$ and $\lambda$ are fixed. Therefore, \eqref{synth} is a convex semidefinite program when $\tau$ and $\lambda$ are fixed that can be solved using a solver like \verb|CVX| \cite{cvx}. A gridded search for the positive scalars $\tau$ and $\lambda$ can be easily performed.% \eqref{e:LMI1} can be fixed. 

If \eqref{synth} is feasible, then the resulting $\mathcal{A}$ is Schur and nonnegative. Detectability of the pair $(A,C)$ is not sufficient for  there to exist a combination of matrices $L$ and $F$ such that $\mathcal{A}$ is Schur and nonnegative.% Consider an example where $A={\bf 1}_{n\times n}$, the matrix of all ones, and $C=\begin{bmatrix}1&0\end{bmatrix}$. Though this pair $(A,C)$ is detectible, there does not exist $L\in\mathbb{R}^{2\times 1}$ and $F\in\mathbb{R}^{2\times 2}$ such that $\mathcal{A}$ is Schur and nonnegative.
\begin{lemma}\label{l:diagschur}
Suppose $\mathcal{A}$, defined in \eqref{e:A}, is Schur and nonnegative. Then $A-LC$ is Schur and $(A-LC)_{ii}\in(-1,1)$ for all $i=1,\dots, n$.
\end{lemma} 
\begin{proof}
It is clear that $A-LC$ is Schur because, according to \cite[Lemma 3]{IPR}, $\sigma(\mathcal{A})=\sigma(A-LC)\cup \sigma(A-LC+2F)$. Now consider two cases:
\begin{enumerate}[1)]
\item Suppose for some $i\in\{1,\dots, n\}$, $(A-LC)_{ii}\ge 1$.  Since $F\ge 0$, $(A-LC+F)_{ii}\ge 1$. Therefore, $(\mathcal{A}-I_{2n})_{ii}\ge 0$. Since $\mathcal{A}\ge 0$ is Schur, $\mathcal{A}-I_{2n}$ is Hurwitz and Metzler  \cite[Corollary 2]{RecursiveMetzler}. This is a contradiction because by \cite[Lemma 3]{poslinobs}, a Metzler matrix is Hurwitz only if its diagonal entries are negative.
\item Suppose for some $i\in\{1,\dots, n\}$, $(A-LC)_{ii}\le -1$. It is then necessary that $F_{ii}\ge 1$, meaning that $(A-LC+2F)_{ii}\ge 1$. Since $A-LC+2F$ is nonnegative, it is not Schur by the same arguments as above. This is a contradiction because $\sigma(\mathcal{A})=\sigma(A-LC)\cup \sigma(A-LC+2F)$ and $\mathcal{A}$ is Schur. 
\end{enumerate}
This concludes the proof.
\end{proof}

It is evident that a necessary condition for feasibility of  \eqref{synth}  is the existence of $\tilde{L}\in\mathbb{R}^{n\times m}$ such that $A-\tilde{L}C$ is Schur and $(A-\tilde{L}C)_{ii}\in(-1,1)$ for all $i=1,\dots, n$. This is not possible for every detectable pair $(A,C)$. The sequel will discuss how to overcome this limitation by using coordinate transformation.

\section{Synthesis of Interval Observers with Coordinate Transformation}\label{s:synthcoord} 
The NDT \eqref{e:sysNDT} is expressed in the coordinates $z=Sx$ as follows:
\begin{align*}
&z[k+1] = SAUz[k]+Sp(Uz[k])+Sw[k],\; y[k] = CUz[k]. 
\end{align*}
The dynamics of $\xi$, defined in \eqref{xidef}, are found:
\begin{align*}
\xi[k+1] = \tilde{\mathcal{A}}\xi[k]+\Delta \pi[k]+\Xi \Delta w[k],
\end{align*}
where 
\begin{align*}
& \tilde{\mathcal{A}} =\small \left[\begin{array}{cc}
S(A-\Lambda C)U+\Phi & \Phi\\
\Phi & S(A-\Lambda C)U+\Phi
\end{array}\right]\normalsize,\\
&\Delta \pi[k]=\small \left[\begin{array}{cc}
\tilde{\pi}\left(\overline{z}[k],\underline{z}[k], y[k]\right)-Sp(Uz[k])\\
Sp(Uz[k])-\tilde{\pi}\left(\underline{z}[k],\overline{z}[k], y[k]\right)
\end{array}\right]\normalsize,\\ &\Xi=\small \begin{bmatrix}S^\oplus & S^\ominus\\ S^\ominus & S^\oplus \end{bmatrix}, \normalsize\nonumber
\end{align*}
and $\Delta w[k]$ is defined in \eqref{e:deltawdef}. The matrices $\Lambda, S, \Phi$ are to be such that $\tilde{\mathcal{A}}$ is Schur and nonnegative. In accordance with Lemma \ref{l:diagschur}, $S$ and $\Lambda$ should be chosen such that the following assumption holds:

\begin{assumption}\label{assSL}
The matrices $\Lambda \in\mathbb{R}^{n\times m}$, $S\in\mathbb{R}^{n\times n}$, and $U=S^{-1}$ are such that $\aleph=S(A-\Lambda C)U$ is Schur and is such that $\aleph_{ii}\in(-1, 1)$ for all $i=1, \dots n$. 
\end{assumption}

There are a couple approaches to construct $\Lambda$ and $S$. If $\Lambda$ is chosen such that $A-\Lambda C$ is Schur and has real eigenvalues, then the Jordan decomposition of $A-\Lambda C$ can be used to construct $S$.  If $\Lambda$ is chosen such that $A-\Lambda C$ is Schur and has complex eigenvalues, then \cite[Lemma 1]{IntervalEstimationNonlin} can be used. Note that $S(A-\Lambda C)U$ does not necessarily have to be nonnegative. 

Once $\Lambda$ and $S$ are chosen, the rest of the matrices required to solve Problem \ref{p:coord} are found from a semidefinite program that is modified from the one described in Theorem \ref{t:syntheorem}.
\begin{corollary}\label{c:bundem}
Consider the system \eqref{e:sysNDT} where the nonlinearity $p$ satisfies Assumption \ref{assp} and Assumption \ref{assvw} holds. Suppose $\Lambda\in\mathbb{R}^{n\times m}$ and $S\in\mathbb{R}^{n\times n}$ are given such that Assumption \ref{assSL} holds. Further, suppose there exists an M-matrix ${\bf J}\in\mathbb{M}^{n\times n}$, matrices ${\bf H}\in\mathbb{R}^{n\times m}$, ${\bf W,\underline{\Upsilon},\overline{\Upsilon},\Gamma}\in\mathbb{R}^{n\times n}_{\ge 0}$, positive definite matrix ${\bf P}={\bf P}^\top\in\mathbb{R}^{2n\times 2n}$, scalars $\gamma, \tau >0$ and $\lambda\in[0,1)$ such that the following hold: \eqref{e:Qpos}, \eqref{e:LMI1}, 
\begin{subequations}
\begin{align}
 -{\bf \underline{\Upsilon}}\le U-{\bf H}CU\le {\bf \overline{\Upsilon}},\label{e:poscoord}\\
\underline{\Theta}\hspace{1pt}{\bf \overline{\Upsilon}}-\overline{\Theta}{\bf \underline{\Upsilon}}+{\bf \Gamma}\ge 0,\label{e:poscoord2}
\end{align}
\end{subequations}
where $U=S^{-1}$, $\mathcal{J}=\text{blkdiag}\left({\bf J,J}\right)$, $\overline{\Theta} = S^\oplus \overline{D}-S^\ominus \underline{D},\; \underline{\Theta} = S^\oplus \underline{D}-S^\ominus \overline{D}$, 
\begin{align}
\mathcal{Q} =\small  \left[\begin{array}{cc}
{\bf J}S(A-\Lambda C)U+{\bf W} & {\bf W}\\ 
{\bf W} & {\bf J}S(A-\Lambda C)U+{\bf W}
\end{array}\right],\normalsize\nonumber\\
\Psi = \small \left[\begin{array}{cc}
\overline{\Theta}\hspace{1pt}{\bf \overline{\Upsilon}}-\underline{\Theta}\hspace{1pt}{\bf \underline{\Upsilon}}+{\bf \Gamma }& {\bf \Gamma}\\
{\bf \Gamma} &\overline{\Theta}\hspace{1pt}{\bf \overline{\Upsilon}}-\underline{\Theta}\hspace{1pt}{\bf \underline{\Upsilon}}+{\bf \Gamma}
\end{array}\right].\normalsize\label{e:newpsi}
\end{align}
The IO \eqref{e:intobscoord} where $\Phi={\bf J}^{-1}{\bf W},  \;H={\bf H}, \;\Gamma={\bf \Gamma},$ and the initial conditions are such that $\underline{z}[0]\le z[0]\le \overline{z}[0]$, satisfies the following:
\begin{enumerate}[i.]
\item $\underline{z}[k]\le z[k]\le \overline{z}[k]$ for all $k\in\mathbb{Z}_{\ge 0}$. 
\item There exist functions $\beta\in\mathcal{KL}$ and $\rho\in\mathcal{K}_\infty$ such that $\left\|\xi[k]\right\|\le \beta\left(\|\xi[0]\|, k\right)+\rho\left(\|\Delta w\|_{\ell_\infty}\right).$
\end{enumerate}
\end{corollary}
\begin{proof}
Using Lemma \ref{l:mmp}, $\underline{\Theta}\le SD\le \overline{\Theta}$ for all $\underline{D}\le D\le \overline{D}$. Moreover, since $S^\oplus, S^\ominus\ge 0$, it is clear that $\underline{\Theta}\le 0$ and $\overline{\Theta}\ge 0$. Hence, Lemma \ref{l:fustercluck} can be used with \eqref{e:poscoord} to show that
\begin{align*}
\underline{\Theta}{\bf \overline{\Upsilon}}-\overline{\Theta}\hspace{1pt}{\bf \underline{\Upsilon}}\le SD(U-{\bf H}CU)\le \overline{\Theta}\hspace{1pt}{\bf \overline{\Upsilon}}-\underline{\Theta}\hspace{1pt}{\bf \underline{\Upsilon}}
\end{align*}
for all $\underline{D}\le D\le \overline{D}$. This can be used to show that \eqref{e:poscoord} and \eqref{e:poscoord2} imply $\Delta \pi[k]\ge 0$ and  $\Delta \pi[k]\le \Psi\xi[k]$, where $\Psi$ is defined in \eqref{e:newpsi}, for all $k\in\mathbb{Z}_{\ge 0}$. The rest of the steps of the proof follow analogously to the proof of Theorem \ref{t:syntheorem}.
\end{proof}

 The synthesis process for the IO \eqref{e:intobscoord}  follows two steps. First, find $\Lambda$ and $S$  to satisfy Assumption \ref{assSL}. Then, solve the semidefinite program in Corollary \ref{c:bundem} for $\Phi$, $H$ and $\Gamma$. Assumption \ref{assSL} is a necessary, but  not sufficient, condition for \eqref{e:intobscoord} to be synthesized using  Corollary \ref{c:bundem}. Therefore, several choices of $\Lambda$ and $S$ may need to be tried. 

If it is known that the initial condition $x[0]$ satisfies $\underline{x}_0\le x[0]\le \overline{x}_0$, then \eqref{e:intobscoord} can be initialized by $\overline{z}[0]=S^\oplus \overline{x}_0-S^\ominus \underline{x}_0$ and $\underline{z}[0]=S^\oplus \underline{x}_0-S^\ominus \overline{x}_0$. Then, since $\underline{z}[k]\le z[k]\le \overline{z}[k]$, $\underline{x}[k]\le x[k[\le \overline{x}[k]$ where 
\begin{align}
\overline{x}[k] = U^\oplus \overline{z}[k]-U^\ominus \underline{z}[k],\; \underline{x}[k] = U^\oplus \underline{z}[k]-U^\ominus \overline{z}[k],\label{e:lll}
\end{align}
by Lemma \ref{l:mmp}. Furthermore, 
\begin{align*}
\varepsilon[k] = \small \begin{bmatrix}U^\oplus & U^\ominus\\ U^\ominus & U^\oplus \end{bmatrix}\xi[k],\normalsize
\end{align*}
where $\varepsilon[k]$ is defined in \eqref{edef}, so the interval defined by the upper bound $\overline{x}[k]$ and lower bound $\underline{x}[k]$ in \eqref{e:lll} remains bounded and the ultimate bound is proportional to $\|\Delta w\|_{\ell_\infty}$. 

\section{Numerical Simulations}\label{s:numex} 
\subsection{Advantage of the Injection Feedback}
This example will serve to elucidate the advantage of using the nonlinear injection feedback by considering a system where 
\begin{align*}
A = \small \begin{bmatrix}1 & 0 \\ 0 & 0\end{bmatrix}\normalsize ,\; C = \small \begin{bmatrix}1 & 0 \end{bmatrix},\normalsize
\end{align*}
and the nonlinearity $p$ is unspecified, but it satisfies Assumption \ref{assp} with
\begin{align*}
\overline{D}=-\underline{D}= \alpha \tilde{D}
\end{align*}
for some $\alpha>0$ and some matrix $\tilde D\ge 0$. The largest value of $\alpha$ such that \eqref{synth} is feasible is characterized for several $\tilde D$ matrices in Table \ref{table} when the injection feedback is used (i.e. $K\ne 0$) and when it is not used (i.e. $K=0$). 

 \begin{table}[!h]
\renewcommand{\arraystretch}{2}
\caption{Maximum $\alpha$ such that \eqref{synth} is feasible with $\overline{D}=-\underline{D}= \alpha \tilde{D}$}
\label{table}
\setlength{\tabcolsep}{3pt}
\centering
\begin{tabular}{|c|c|c|c|c|c|c|c|c|}
\hline 
$\tilde D $ &
$\begin{bmatrix}
0 & 1\\ 1 & 0
\end{bmatrix}$ & $\begin{bmatrix}
1 & 1\\ 1 & 0
\end{bmatrix}$ & $\begin{bmatrix}
1 & 1\\ 0 & 0
\end{bmatrix}$ & $\begin{bmatrix}
0 & 0\\ 1 & 1
\end{bmatrix}$ 
 & $\begin{bmatrix}
0 & 1\\ 1 & 1
\end{bmatrix}$ 
 & $\begin{bmatrix}
1 & 1\\ 1 & 1
\end{bmatrix}$\\
\hline
$K=0$ & 0.33 & 0.20 & 0.27 & 0.27 & 0.16 & 0.20 \\
\hline
$K\ne 0$ & 0.66 & 0.66 & 0.66  & 0.33 & 0.27 & 0.27\\ 
\hline 
\end{tabular}
\label{tab1}
\end{table}

% \begin{table}[!h]
%\renewcommand{\arraystretch}{2}
%\caption{Maximum $\alpha$ such that \eqref{synth} is feasible with $\overline{D}=-\underline{D}= \alpha \tilde{D}$}
%\label{table}
%\setlength{\tabcolsep}{3pt}
%\centering
%\begin{tabular}{|c|c|c|}
%\hline 
%$\tilde D $ &
%$K=0$ &
%$K\ne 0$ \\
% \hline 
%$\begin{bmatrix}
%0 & 1\\ 1 & 0
%\end{bmatrix}$& 0.33  & 0.66 \\
%\hline 
%$\begin{bmatrix}
%1 & 1\\ 1 & 0
%\end{bmatrix}$ & 0.20  & 0.66 \\
%\hline 
%$\begin{bmatrix}
%1 & 1\\ 0 & 0
%\end{bmatrix}$ & 0.27  & 0.66\\
%\hline
%$\begin{bmatrix}
% 0 & 0\\ 1 & 1
%\end{bmatrix}$ & 0.27 & 0.33\\
%\hline
%$\begin{bmatrix}
% 1 & 1\\ 1 & 1
%\end{bmatrix}$ & 0.16 & 0.27\\
%\hline
%$\begin{bmatrix}
% 0 & 1\\ 1 & 1
%\end{bmatrix}$ & 0.20 & 0.27\\
%\hline
%\end{tabular}
%\label{tab1}
%\end{table}

Consistently, the larger maximum values of $\alpha$ occur when the injection feedback is used. In most cases, the maximum $\alpha$ is nearly doubled or more. This demonstrates that using the injection terms allows for IOs to be synthesized for systems with nonlinearities that have larger variations.

%\begin{table}[!h]
%\renewcommand{\arraystretch}{2}
%\caption{Maximum Feasible $\alpha$ when $\overline{D}=-\underline{D}= \alpha \tilde{D}$}
%\label{table}
%\setlength{\tabcolsep}{3pt}
%\centering
%\begin{tabular}{|c|c|c|c|}
%%\hline 
%%$\tilde D $ &
%%$K,K=0$ &
%%$K,K\ne 0$ \\
%% \hline 
%$\begin{bmatrix}
%0 & 1\\ 1 & 0
%\end{bmatrix}$ &
%$\begin{bmatrix}
%1 & 1\\ 1 & 0
%\end{bmatrix}$  
%%$\begin{bmatrix}
%%1 & 1\\ 0 & 0
%%\end{bmatrix}$ & 
%%$\begin{bmatrix}
%% 0 & 0\\ 1 & 1
%%\end{bmatrix}$ &
%%$\begin{bmatrix}
%% 1 & 1\\ 1 & 1
%%\end{bmatrix}$ & 
%%$\begin{bmatrix}
%% 0 & 1\\ 1 & 1
%%\end{bmatrix}$ 
%\end{tabular}
%\label{tab1}
%\end{table}

\subsection{Nonlinear System with Sampled Output}
There are many results in the literature for the synthesis of IOs for linear continuous-time systems with sampled outputs \cite{intervalwithdtmeasurement, CDobserver, SampledIO}, and fewer results for NCTs with sampled outputs \cite{Me, CDIOMicroalgae}. The difficulty dealing with nonlinear systems is the impracticability of capturing the inter-sampling behavior exactly. This example uses a forward-Euler approximation for guaranteed interval state estimation of a nonlinear  pendulum model with sampled output:
\begin{subequations}\label{e:conttime}
\begin{align}
\dot{x}(t) &= A_c x(t) +p_c(x(t)),\; t\in\mathbb{R}_{\ge 0}\\
y(kh)&=Cx(kh),\; k\in\mathbb{Z}_{\ge 0},
\end{align}
\end{subequations}
where $x_1$ is the position of the pendulum, $x_2$ is its angular velocity, $h$ is the sampling time,
\begin{align*}
A_c = \small \begin{bmatrix}0 & 1\\ 0 & 0\end{bmatrix}\normalsize,\; p_c(x)=\small\begin{bmatrix}0\\ -\sin(x_2)\end{bmatrix},\; C=\begin{bmatrix}1 & 0\end{bmatrix}.\normalsize
\end{align*}

 The exact discretization of \eqref{e:conttime} is the following:
\begin{align*}
x[k+1]=F_h^e(x[k]),\; y[k] = Cx[k].
\end{align*}
The function $F_h^e(x)$ is the exact state-transition matrix between samples. It is unknown, but can be approximated by forward-Euler as an NDT in the form of \eqref{e:sysNDT}, where
\begin{align*}
A = I_n+h\cdot A_c,\; p(x[k])=h\cdot p_c(x[k]), 
\end{align*}
and $w[k]$ is the approximation error: 
\begin{align*}
w[k] = F_h^e(x[k])-(Ax[k]+p(x[k])),
\end{align*}
which is unknown. As $p_c$ is globally Lipschitz, the forward-Euler approximation is consistent with the exact discretization \cite{approxobserver}. This means that if $x[k]$ belongs to a compact set $\mathcal{X}$ for all $k\in\mathbb{Z}_{\ge 0}$ and $h$ is sufficiently small, there exists $\varrho\in\mathcal{K}$ such that 
\begin{align*}
\|w[k]\|\le h\varrho(h),
\end{align*}
 for all $k\in\mathbb{Z}_{\ge 0}$. Consequently, Assumption \ref{assvw} holds where $\overline{w}=-\underline{w}=h\varrho(h){\bf 1}_2$ and ${\bf 1}_2$ is the vector of ones in $\mathbb{R}^2$. For this example, it can be deduced from the details of the proof of \cite[Lemma 1]{DsicreteApprox} that for $\mathcal{X}=[-\frac{\pi}{2},\frac{\pi}{2}]\times[-1, 1]$, $\varrho(h)=\sqrt{2}h$. 

For the rest of this example, let $h=0.065$ secs. For all $\tilde{L}\in\mathbb{R}^{2}$, $(A-\tilde{L}C)_{22}=1$, so the IO \eqref{e:intobs} cannot be synthesized using Theorem \ref{t:syntheorem}. Alternatively, consider 
\begin{align*}
\Lambda = \small \begin{bmatrix}0.9 & 0.5\end{bmatrix}^\top,\normalsize\; S =\small \begin{bmatrix} 0.6063  & -0.0457\\
   -0.6063  &  1.0457\end{bmatrix},\normalsize
\end{align*}
which satisfies Assumption \ref{assSL}. The transformation matrix $S$ is found using the Jordan decomposition. Using Corollary \ref{c:bundem}, the rest of the matrices for the IO \eqref{e:intobscoord} are found:
\begin{align*}
H = \small \begin{bmatrix}1 & 0.5798\end{bmatrix}^\top,\normalsize \; F=0,\; G=0.
\end{align*}
A simulation of the IO \eqref{e:intobscoord} is performed and shown in Fig. \ref{f:pendobs} where $\overline{x}[k]$ and $\underline{x}[k]$ are determined from \eqref{e:lll}. It can be seen that $\underline{x}[k]\le x[k]\le\overline{x}[k]$ from the fact that the blue markers are always above the black lines and the magenta markers are always below the black lines. The ultimate bound on the interval is larger when $h$ is larger and smaller when $h$ is smaller, as expected. 
\begin{figure}[!t]
\centerline{\includegraphics[width=1.05\columnwidth]{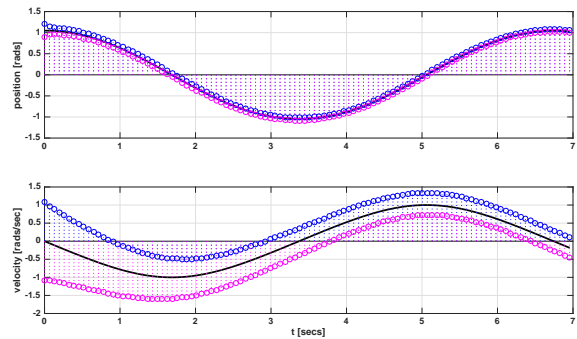}}
\caption{Simulation of the IO \eqref{e:intobscoord} for the nonlinear pendulum example \eqref{e:conttime} with $h=0.065$ secs. The blue and magenta markers denote $\overline{x}$ and $\underline{x}$, respectively. Moreover, dotted lines between the markers and the time axes are provided to show the time of each sample. The black lines are the states of the pendulum. At each sampling instance $t=kh$, the interval estimate for the next sampling instance $t=(k+1)h$ is determined. No information about the state between sampling instances is provided by the IO \eqref{e:intobscoord}; however, intervals for the state between samples can be found by open-loop integration.}
\label{f:pendobs}
\end{figure}
\section{Conclusions}
In this paper, Theorem \ref{t:syntheorem} and Corollary \ref{c:bundem} provide semidefinite feasibility programs to synthesize IOs without and with coordinate transformation, respectively. Lemma \ref{l:diagschur} provides the conditions under which coordinate transformations are necessary to synthesize an IO.  To the best of the authors knowledge, there are no other results in the literature where nonlinear injection feedback terms for IOs are synthesized. It is demonstrated through example that the use of nonlinear injection feedback allows for IOs to be synthesized when the nonlinearities have larger variation. Further, the application to NCTs with sampled output is demonstrated through example. Future work should be dedicated to IO-based feedback control and the synthesis of IOs for NDTs with delays. 

 \bibliography{DTBib}

\begin{thebibliography}{10}
\providecommand{\url}[1]{#1}
\csname url@samestyle\endcsname
\providecommand{\newblock}{\relax}
\providecommand{\bibinfo}[2]{#2}
\providecommand{\BIBentrySTDinterwordspacing}{\spaceskip=0pt\relax}
\providecommand{\BIBentryALTinterwordstretchfactor}{4}
\providecommand{\BIBentryALTinterwordspacing}{\spaceskip=\fontdimen2\font plus
\BIBentryALTinterwordstretchfactor\fontdimen3\font minus
  \fontdimen4\font\relax}
\providecommand{\BIBforeignlanguage}[2]{{%
\expandafter\ifx\csname l@#1\endcsname\relax
\typeout{** WARNING: IEEEtran.bst: No hyphenation pattern has been}%
\typeout{** loaded for the language `#1'. Using the pattern for}%
\typeout{** the default language instead.}%
\else
\language=\csname l@#1\endcsname
\fi
#2}}
\providecommand{\BIBdecl}{\relax}
\BIBdecl

\bibitem{astromwitenmark}
K.~J. {\AA}str{\"{o}}m and B.~Wittenmark, \emph{Computer Controlled Systems:
  Theory and Design}.\hskip 1em plus 0.5em minus 0.4em\relax Englewood Cliffs,
  NJ: Pretence-Hall, 1997.

\bibitem{disctaylor}
N.~Kazantzis and C.~Kravaris, ``Time-discretization of nonlinear control
  systems via {Taylor} methods,'' \emph{Computers \& Chemical Engineering},
  vol.~23, pp. 763--784, 1999.

\bibitem{approxobserver}
M.~Ar\c{c}ak and D.~Nesi{\'{c}}, ``A framework for nonlinear sampled-data
  observer design via approximate discrete-time models and emulation,''
  \emph{Automatica}, vol.~40, pp. 1931--1938, 2004.

\bibitem{continuousidentification}
H.~Unbehauen and G.~P. Rao, ``Continuous-time approaches to system
  identification--a survey,'' \emph{Automatica}, vol.~26, no.~1, pp. 23--35,
  1990.

\bibitem{EKFNonDisc}
K.~Reif and R.~Unbehauen, ``The extended {K}alman filter as an exponential
  observer for nonlinear systems,'' \emph{IEEE Transactions on Signal
  Processing}, vol.~47, no.~8, pp. 2324--2328, 1999.

\bibitem{UnscentedDT}
S.~J. Julier, J.~K. Uhlmann, and H.~F. Durrant-Whyte, ``A new method for the
  nonlinear transformation of means and covariances in filters and
  estimators,'' \emph{IEEE Transactions on Automatic Control}, vol.~45, pp.
  477--482, 2000.

\bibitem{high_gainobsdisc}
A.~M. Dabroom and H.~K. Khalil, ``Output feedback sampled-data control of
  nonlinear systems using high-gain observers,'' \emph{IEEE Transactions on
  Automatic Control}, vol.~46, no.~11, pp. 1712--1725, 2001.

\bibitem{linearization}
W.~Lin and C.~I. Byrnes, ``Remarks on linearization of discrete-time autonomous
  systems and nonlinear observer design,'' \emph{Systems \& Control Letters},
  vol.~25, pp. 31--40, 1995.

\bibitem{linearizableerror}
M.~Xiao, N.~Kazantzis, C.~Kravaris, and A.~J. Krener, ``Nonlinear discrete-time
  observer design with linearizable error dynamics,'' \emph{IEEE Transactions
  on Automatic Control}, vol.~48, no.~4, pp. 622--626, 2003.

\bibitem{ObsDT}
C.~Califano, S.~Monaco, and D.~Normand-Cyrot, ``On the observer design in
  discrete-time,'' \emph{Systems \& Control Letters}, vol.~49, pp. 255--265,
  2003.

\bibitem{movinghorizon}
C.~V. Rao, J.~B. Rawlings, and D.~Q. Mayne, ``Constrained state estimation for
  nonlinear discrete-time systems: stability and moving horizon
  approximations,'' \emph{IEEE Transactions on Automatic Control}, pp.
  246--258, 2003.

\bibitem{ObserverLipDisc}
G.~I. Bara, A.~Zemouche, and M.~Boutayeb, ``Observer synthesis for {L}ipschitz
  discrete-time systems,'' in \emph{Proceedings of International Symposium on
  Circuits and Systems}, Kobe, Japan, 2005, pp. 3195--3198.

\bibitem{discretetimeluenberger}
A.~Zemouche and M.~Boutayeb, ``Observer design for {L}ipschitz nonlinear
  systems: The discrete-time case,'' \emph{IEEE Transactions on Circuits and
  Systems II: Express Briefs}, vol.~53, no.~8, pp. 777--781, 2006.

\bibitem{UnknownInputObs}
A.~Chakrabarty, S.~H. \.{Z}ak, and S.~Sundaram, ``State and unknown input
  observers for discrete-time nonlinear systems,'' in \emph{Proceedings of the
  Conference on Decision and Control}, Las Vegas, NV, 2016, pp. 7111--7116.

\bibitem{pigeons}
W.~Zhang, Y.~Zhao, M.~Abbaszadeh, M.~Ji, and X.~Cai, ``Exponential observers
  for discrete-time nonlinear systems with incremental quadratic constraints,''
  in \emph{Proceedings of the American Control Conference}, Philadelphia, PA,
  2019, pp. 477--482.

\bibitem{luenbergerdt}
L.~Brivadis, V.~Andrieu, and U.~Serres, ``Luenberger observers for
  discrete-time nonlinear systems,'' in \emph{Proceedings of the Conference on
  Decision and Control}, Nice, France, 2019, pp. 3435--3440.

\bibitem{DiscreteIO}
D.~Efimov, W.~Perruquetti, T.~Ra{\"{i}}ssi, and A.~Zolghadri, ``Interval
  observers for time-varying discrete-time systems,'' \emph{IEEE Transactions
  on Automatic Control}, vol.~58, no.~12, pp. 3218--3224, 2013.

\bibitem{IODT}
F.~Mazenc, T.~N. Dinh, and S.~I. Niculescu, ``Interval observers for
  discrete-time systems,'' \emph{International Journal of Robust and Nonlinear
  Control}, vol.~24, pp. 2867--2890, 2014.

\bibitem{UncertainDiscreteTime}
Z.~Wang, C.-C. Lim, and Y.~Shen, ``Interval observer design for uncertain
  discrete-time linear systems,'' \emph{Systems \& Control Letters}, vol. 116,
  pp. 41--46, 2018.

\bibitem{Outbreakobserver}
K.~H. Degue and J.~{Le Ny}, ``Estimation and outbreak detection with interval
  observers for uncertain discrete-time {SEIR} epidemic models,''
  \emph{International Journal of Control}, 2019.

\bibitem{Polytopic}
A.~M. Tahir, X.~Xu, and B.~A\c{c}{\i}kme\c{s}e, ``Synthesis of interval
  observers for polytopic systems and conic systems,'' in \emph{Proceedings of
  the Conference on Decision and Control}, Nice, France, 2019, pp. 3447--3452.

\bibitem{IntervalEstimationNonlin}
T.~Ra{\"{i}}ssi, D.~Efimov, and A.~Zolghadri, ``Interval state estimation for a
  class of nonlinear systems,'' \emph{IEEE Transactions on Automatic Control},
  vol.~57, no.~1, pp. 260--265, 2012.

\bibitem{IntervalLip}
M.~Moisan and O.~Bernard, ``Robust interval observers for global {L}ipschitz
  uncertain chaotic systems,'' \emph{Systems \& Control Letters}, vol.~59, pp.
  687--694, 2010.

\bibitem{TahirThesis}
A.~M. Tahir, ``Estimation and control of nonlinear hybrid systems and nonaffine
  control,'' Ph.D. dissertation, University of Washington, Seattle, WA, 2019.

\bibitem{CooganFinitieAbstraction}
S.~Coogan and M.~Ar\c{c}ak, ``Efficient finite abstraction of mixed monotone
  systems,'' in \emph{Proceedings of the 18th International Conference on
  Hybrid Systems: Computation and Control}, Seattle, WA, 2015, pp. 58--67.

\bibitem{SuffMixMonotone}
L.~Yang, O.~Mickelin, and N.~Ozay, ``On sufficient conditions for mixed
  monotonicity,'' \emph{IEEE Transactions on Automatic Control}, vol.~64,
  no.~12, pp. 5080 -- 5085, 2019.

\bibitem{circlecriterion}
M.~Ar\c{c}ak and P.~Kokotovi\'{c}, ``Nonlinear observers: a circle criterion
  design and robustness analysis,'' \emph{Automatica}, vol.~37, no.~12, pp.
  1923--1930, 2001.

\bibitem{DQCObserver}
B.~A\c{c}{\i}kme\c{s}e and M.~J. Corless, ``Observers for systems with
  nonlinearities satisfying incremental quadratic contraints,''
  \emph{Automatica}, vol.~47, pp. 1339--1348, 2011.

\bibitem{PNLV}
Y.~Wang, R.~Rajamani, and D.~M. Bevly, ``Observer design for parameter varying
  differentiable nonlinear systems, with application to slip angle
  estimation,'' \emph{IEEE Transactions on Automatic Control}, vol.~62, no.~4,
  pp. 1940--1945, 2017.

\bibitem{Outputfeedback}
X.~Xu, B.~A\c{c}{\i}kme\c{s}e, and M.~J. Corless, ``Observer-based controllers
  for incrementally quadratic nonlinear systems with disturbances:
  continuous-time and event-triggered cases,'' \emph{IEEE Transactions on
  Automatic Control (accepted))}, 2020.

\bibitem{ISSdiscrete}
Z.-P. Jiang and Y.~Wang, ``Input-to-state stability for discrete-time nonlinear
  systems,'' \emph{Automatica}, vol.~37, pp. 857--869, 2001.

\bibitem{Greenbaum}
A.~Greenbaum, \emph{Iterative Methods for Solving Linear Systems}.\hskip 1em
  plus 0.5em minus 0.4em\relax Philadelphia, PA: SIAM Frontiers in Mathematics,
  1997.

\bibitem{LMIbook}
S.~P. Boyd, L.~{El Ghaoui}, E.~Feron, and V.~Balakrishnan, \emph{Linear Matrix
  Inequalities in System and Control Theory}.\hskip 1em plus 0.5em minus
  0.4em\relax Philadelphia, PA: SIAM Studies in Applied Mathematics, 1994.

\bibitem{LMIcond}
G.~Pipeleers, B.~Demeulenaere, J.~Swevers, and L.~Vandenberghe, ``Extended
  {LMI} characterizations for stability and performance of linear systems,''
  \emph{Systems \& Control Letters}, vol.~58, pp. 510--518, 2008.

\bibitem{cvx}
M.~Grant and S.~P. Boyd, ``{CVX}: Matlab software for disciplined convex
  programming, version 2.1,'' {http://cvxr.com/cvx}.

\bibitem{IPR}
F.~Cacace, L.~Farina, A.~Germani, and C.~Manes, ``Internally positive
  representation of a class of continuous time systems,'' \emph{IEEE
  Transactions on Automatic Control}, vol.~57, no.~12, pp. 3158--3163, 2012.

\bibitem{RecursiveMetzler}
M.~Souza, F.~R. Wirth, and R.~N. Shorten, ``A note on recursive {S}chur
  complements, block {H}urwitz stability of {M}etzler matrices, and related
  results,'' \emph{IEEE Transactions on Automatic Control}, vol.~62, no.~8, pp.
  4167--4172, 2017.

\bibitem{poslinobs}
N.~Dautrebande and G.~Bastin, ``Positive linear observers for positive linear
  systems,'' in \emph{Proceedings of the European Control Conference},
  Karlsruhe, Germany, 1999.

\bibitem{intervalwithdtmeasurement}
F.~Mazenc, M.~Kieffer, and {\'{E}}.~Walter, ``Interval observers for
  continuous-time linear systems with discrete-time outputs,'' in
  \emph{Proceedings of the American Control Conference}, Montr\'{e}al, QC,
  2012, pp. 1889--1894.

\bibitem{CDobserver}
F.~Mazenc and T.~N. Dinh, ``Construction of interval observers for
  continuous-time systems with discrete measurements,'' \emph{Automatica},
  vol.~50, pp. 2555--2560, 2014.

\bibitem{SampledIO}
D.~Efimov, E.~Fridman, A.~Polyakov, W.~Perruquetti, and J.-P. Richard, ``On
  design of interval observers with sampled measurement,'' \emph{Systems \&
  Control Letters}, vol.~96, pp. 158--164, 2016.

\bibitem{Me}
A.~M. Tahir, X.~Xu, and B.~A\c{c}{\i}kme\c{s}e, ``Self-triggered interval
  observers for {L}ipschitz nonlinear systems,'' in \emph{Proceedings of the
  American Control Conference}, Philadelphia, PA, 2019, pp. 465--470.

\bibitem{CDIOMicroalgae}
G.~Goffaux, A.~{Vande Wouwer}, and O.~Bernard, ``Continuous-discrete interval
  observers for monitoring microalgae cultures,'' \emph{Biotechnology
  Progress}, vol.~25, no.~3, pp. 667--675, 2009.

\bibitem{DsicreteApprox}
D.~Nesi{\'{c}}, A.~R. Teel, and P.~Kokotovi\'{c}, ``Sufficient conditions for
  stabilization of sampled-data nonlinear systems via discrete-time
  approximation,'' \emph{Systems \& Control Letters}, vol.~38, pp. 259--270,
  1999.

\end{thebibliography}

\end{document}